\documentclass[a4paper]{article}

\usepackage{amsmath}
\usepackage{amssymb}
\usepackage{amsfonts}
\usepackage{amsthm}
\usepackage{mathbbol}
\usepackage{epsf}
\usepackage{epsfig}
\usepackage{subfigure}
\usepackage{color}

\usepackage[colorlinks=true]{hyperref}
\hypersetup{citecolor=blue}

\newtheorem{Lem}{Lemma}[section]
\newtheorem{The}{Theorem}[section]
\newtheorem{Pro}{Proposition}[section]

\newtheorem{Rem}{Remark}[section]

\DeclareMathOperator*{\esssup}{ess\,sup}

\newcommand{\R}{\mathbb R}                        
\newcommand{\fref}[1]{\eqref{#1}}                 
\renewcommand{\phi}{\varphi}                      
\newcommand{\eps}{\varepsilon}                    
\newcommand{\cd}{\!\cdot\!}                       

\newcommand{\da}{d_1}                             
\newcommand{\db}{d_2}                             
\newcommand{\dc}{d_3}                             
\newcommand{\lap}{\triangle_{x}}                  
\newcommand{\nx}{\nabla_{\!\!x}\,}                  
\newcommand{\iO}{\int_\Omega}                     
\newcommand{\ai}{u_{\infty}}                      
\newcommand{\aq}{\overline{u}}                    
\newcommand{\Ai}{U_{\infty}}                      
\newcommand{\Aq}{\overline{U}}                    
\newcommand{\Aqq}{\overline{U^2}}                 
\newcommand{\bi}{v_{\infty}}                      
\newcommand{\bq}{\overline{v}}                    
\newcommand{\Bi}{V_{\infty}}                      
\newcommand{\Bq}{\overline{V}}                    
\newcommand{\Bqq}{\overline{V^2}}                 
\newcommand{\ci}{w_{\infty}}                      
\newcommand{\cq}{\overline{w}}                    
\newcommand{\Ci}{W_{\infty}}                      
\newcommand{\Cq}{\overline{W}}                    
\newcommand{\Cqq}{\overline{W^2}}                 
\newcommand{\ma}{\mu_1}                           
\newcommand{\mcu}{\mu_{3,max}}                    
\newcommand{\mb}{\mu_2}                           
\newcommand{\mc}{\mu_3}                           
\newcommand{\dela}{\delta_1}                      
\newcommand{\delb}{\delta_2}                      
\newcommand{\delc}{\delta_3}                      
\def\pa{\partial}
\def\ov{\overline}

\title{Exponential decay towards equilibrium
and global classical solutions for nonlinear reaction-diffusion systems}

\author{Klemens Fellner\footnote{Institut f\"ur Mathematik und
    Wissenschaftliches Rechnen, Heinrichstra{\ss}e 36, 8010
    Graz. Email: \texttt{klemens.fellner@uni-graz.at}.} \and El-Haj Laamri\footnote{Institute Elie Cartan, Universit\'{e} Lorraine, P. B. 239, 54506 Vand{\oe}uvre l\'{e}s Nancy, France E-mail: \texttt{El-Haj.Laamri@univ-lorraine.fr}.} 
}

\begin{document}

\maketitle

\begin{abstract}
We consider a system of reaction-diffusion equations describing 
the reversible reaction of two species $\mathcal{U}, \mathcal{V}$ forming a third species $\mathcal{W}$ and vice versa 
according to mass action law kinetics with arbitrary stochiometric coefficients (equal or larger than one). 

Firstly, we prove existence of global classical solutions via improved duality estimates under the assumption 
that one of the diffusion coefficients of $\mathcal{U}$ or  $\mathcal{V}$
is sufficiently close to the diffusion coefficient of $\mathcal{W}$.

Secondly, we derive an entropy entropy-dissipation estimate, that is a functional inequality, 
which applied to global solutions of these reaction-diffusion system proves 
exponential convergence to equilibrium with explicit rates and constants.  
\end{abstract}
\medskip


\noindent{\bf Acknowledgement:} K.F gratefully acknowledges the hospitality of the 
Institut \'Elie Cartan de Lorraine, Nancy and the partial support of NAWI Graz.
E.-H.L  gratefully acknowledges the hospitality of the Wolfgang Pauli Institute, Vienna and the Institute for Mathematics and Scientific Computing, Graz University. Both authors would like to kindly thank the reviewer, whose careful report was very helpful in improving the manuscript.
\vfil\eject

\section{Introduction}
The aim of this paper is to investigate
the following class of systems of three nonlinear reaction-diffusion equations. 
For stoichiometric coefficients $\alpha, \beta, \gamma \ge 1$, we consider the system  
\begin{align}
&\pa_t u - \da\lap u = -\alpha(\ell\,u^{\alpha}v^{\beta} -k\,w^{\gamma}), \label{ABCA} \\
&\pa_t v - \db\lap v = -\beta(\ell\,u^{\alpha}v^{\beta} - k\,w^{\gamma}), \label{ABCB} \\
&\pa_t w - \dc\lap w = \gamma(\ell\,u^{\alpha}v^{\beta} -k\,w^{\gamma}), \label{ABCC}
\end{align}
where the constants $\ell,k>0$ are positive reaction rates and $\da,\db,\dc>0$
denote positive diffusion coefficients. Moreover, we assume that the concentrations $u,v$ and $w$ satisfy homogeneous Neumann conditions 
\begin{equation}
n(x)\cd\nx u = 0, \qquad n(x)\cd\nx v =0, 
\qquad n(x)\cd\nx w=0 \qquad x\in\partial\Omega ,
\label{ABCBC}
\end{equation}
on a bounded domain $\Omega$ with sufficiently smooth boundary $\partial\Omega$ 
and the non-negative initial condition
\begin{equation}\label{ABCID}
u(0,x) = u_0(x)\ge0, \qquad v(0,x) = v_0(x)\ge0, \qquad w(0,x) = w_0(x)\ge0\,.
\end{equation}

We remark that the reaction dynamics of System \eqref{ABCA}--\eqref{ABCC} 
is positivity preserving (see e.g. \cite{P2}) and, thus, that non-negative initial data $\eqref{ABCID}$ ensures non-negative solutions $u$, $v$, and $w$, i.e.
$$
(P) 
\quad u(t,x)\ge  0, \quad v(t,x)\geq 0, \quad w(t,x)\geq 0, \qquad \forall t\ge 0, \ x\in\Omega. 
$$

Moreover, the System \eqref{ABCA}--\eqref{ABCC} with homogeneous Neumann boundary conditions \eqref{ABCBC} 
satisfies the following two mass conservation laws: By taking the sums  $\gamma\times\eqref{ABCA} +  \alpha\times \eqref{ABCC}$ and $\gamma\times \eqref{ABCB}+  \beta\times \eqref{ABCC}$ and by
integrating by parts with non-flux boundary conditions, we obtain 
the following conservation laws to hold for  solutions of  \eqref{ABCA} -- \eqref{ABCID}:
\begin{align}
\iO (\gamma u(t,x)+\alpha w(t,x))\,dx = 
\iO (\gamma u_0(x)+\alpha w_0(x))\,dx &=: M_1 > 0,  \nonumber \\
\iO (\gamma v(t,x)+\beta w(t,x))\,dx = 
\iO (\gamma v_0(x)+\beta w_0(x))\,dx &=: M_2 > 0  \label{ABCCo},
\end{align}
where we assume strictly positive masses $M_1, M_2>0$.
\medskip
%
%

We remark that the System \eqref{ABCA}--\eqref{ABCC} can be further simplified 
by rescaling times, space and concentrations. In particular for all stoichiometric coefficients
$\alpha+\beta\neq\gamma$, the rescaling
$$
t\rightarrow \Bigl(\frac{k}{\ell}\Bigr)^{\frac{1-\gamma}{\alpha+\beta-\gamma}}\frac{1}{k}\,t, \qquad 
x\rightarrow |\Omega|^{\frac{1}{N}}\,x, \qquad 
(u,v,w)\rightarrow \Bigl(\frac{k}{\ell}\Bigr)^{\frac{1}{\alpha+\beta-\gamma}}(u,v,w),
$$ 
allows (without loss of generality) to consider the System 
\eqref{ABCA} -- \eqref{ABCC} with normalised reaction rates $\ell=k=1$ on a normalised domain, 
i.e. we can replace in all three equations \eqref{ABCA} -- \eqref{ABCC}
\begin{equation*}
(\ell\,u^{\alpha}v^{\beta} - k\, \,w^{\gamma}) \rightarrow (u^{\alpha}v^{\beta} - \,w^{\gamma}),\qquad |\Omega|=1\,. 
\end{equation*}

In the special cases when $\alpha+\beta=\gamma$, it is still possible, for instance, to 
rescaling $(\ell\,u^{\alpha}v^{\beta} - k\, \,w^{\gamma}) \rightarrow (u^{\alpha}v^{\beta} - \,w^{\gamma})$
in the first two equations \eqref{ABCA} and \eqref{ABCB}, but the third equation \eqref{ABCC}
can only be rescaled up to a common factor $\bigl(\frac{\ell}{k}\bigr)^{{1/\gamma}}$, i.e.  
$$
\gamma (\ell\,u^{\alpha}v^{\beta} - k\, \,w^{\gamma}) \rightarrow 
\gamma \,\Bigl(\frac{\ell}{k}\Bigr)^{\frac{1}{\gamma}}(u^{\alpha}v^{\beta} - \,w^{\gamma}).
$$ 
 
For the rest of the article, we shall always consider the complete rescaled system 
\begin{align}
&\pa_t u - \da\lap u = -\alpha(u^{\alpha}v^{\beta} -w^{\gamma}), \label{ABCa} \\
&\pa_t v - \db\lap v = -\beta(u^{\alpha}v^{\beta} - w^{\gamma}), \label{ABCb} \\
&\pa_t w - \dc\lap w = \gamma(u^{\alpha}v^{\beta} -w^{\gamma}),\label{ABCc}
\end{align}
{subject to the initial date \eqref{ABCID} and the homogeneous Neumann boundary conditions \eqref{ABCBC}.}
In the spacial cases $\alpha+\beta=\gamma$, the additional factor $\bigl(\frac{\ell}{k}\bigr)^{{1/\gamma}}$ 
in \eqref{ABCc} leads to non-essential modifications of the arguments and calculations in this paper. 
\medskip

The System \eqref{ABCa} -- \eqref{ABCc} with the conservation laws \eqref{ABCCo} 
possesses a unique detailed balance equilibrium state, which are the unique positive constants ($\ai,\bi,\ci$) balancing the three reversible reactions, i.e. 
$$
\ai^{\alpha}\,\bi^{\beta}=\ci^{\gamma}
$$
and satisfying the conservation laws \eqref{ABCCo}, i.e. 
$$\gamma\ai + \alpha\ci = M_1,\qquad \gamma\bi + \beta\ci  = M_2
$$ 
(recall $|\Omega|=1$). In fact, the uniqueness of the equilibrium ($\ai,\bi,\ci$) follows from the uniqueness of the solution $\ci$ of the transcendental equation
\begin{equation}
\left(\frac{M_1}{\gamma}-\frac{\ci\alpha}{\gamma}\right)^{\alpha}\,\left(\frac{M_2}{\gamma}-\frac{\ci\beta}{\gamma}\right)^{\beta}=\ci^{\gamma},
\label{39bis}
\end{equation}
where the left-hand-side of \eqref{39bis} is strictly monotone decreasing in $\ci$ while the right-hand-side of \eqref{39bis} is strictly monotone increasing in $\ci$.
\medskip 

Finally, the System \eqref{ABCa} -- \eqref{ABCc} 
features a Boltzmann-type entropy functional (or physically, a free energy functional)
\begin{equation}
E(u,v,w)(t) = \iO \left(u(\ln(u)-1)+v(\ln(v)-1)+
w(\ln(w)-1) \right)dx, \label{ABCEntr}
\end{equation}
which dissipates according to the non-negative 
entropy dissipation functional $D=-\frac{d}{dt}E$:
\begin{multline}
D(u,v,w)(t) = \da\iO\frac{|\nx u|^2}{u}\,dx 
+\db\iO\frac{|\nx v|^2}{v}\,dx +\dc\iO\frac{|\nx w|^2}{w}\,dx\\
+\iO(u^{\alpha} v^{\beta}-w^{\gamma})\ln\Bigl(\frac{u^{\alpha} v^{\beta}}{w^{\gamma}}\Bigr)\,dx \ge 0. 
\label{ABCEntrDiss}
\end{multline}
We remark, that thanks to the properties of the equilibrium ($\ai,\bi,\ci$), one can rewrite the 
relative entropy with respect to the equilibrium state in the following way:
\begin{multline}
E(u,v,w)(t) - E(\ai,\bi,\ci) = \iO \left( u\ln\Bigl(\frac{u}{\ai}\Bigr)-(u-\ai)\right.\\
\left.+ v\ln\Bigl(\frac{v}{\bi}\Bigr)-(v-\bi)+
w \ln\Bigl(\frac{w}{\ci}\Bigr)-(w-\ci) \right)dx, \label{ABCRelEntr}
\end{multline}
Clearly, the relative entropy dissipates according to the same entropy dissipation functional, i.e. 
\begin{equation*}
\frac{d}{dt}E(u,v,w)(t) = \frac{d}{dt}(E(u,v,w)-E(\ai,\bi,\ci))(t) = -D(u,v,w)(t).
\end{equation*}

\begin{Rem}
We emphasise that the entropy dissipation functional $D(u,v,w)$ vanishes for all constant states ($u,v,w$)=($\aq,\bq,\cq$), which balance the reactions, i.e. $\aq^{\alpha} \bq^{\beta}=\cq^{\gamma}$. These states, however, form a two parameter family and one crucially requires the two conservation laws \eqref{ABCCo} as additional constraints in order to identify 
the equilibrium ($\ai,\bi,\ci$)
as the unique state, which minimises/annihilates the entropy dissipation functional along the flow of mass-conserving solutions.
\end{Rem}

The following Theorem is the first main result of this paper. 
{The proof is based on applying (point-wise in time) a functional inequality between relative entropy and entropy-dissipation, see Proposition \ref{ABCDDH} below :}

\begin{The} \label{tt2} 
Let $\Omega$ be a 
{
connected open set of $\R^N$ $(N\ge 1)$ with
sufficiently smooth boundary $\partial\Omega$  such that Poincar\'e's inequality and the Logarithmic Sobolev inequality hold.} Let $\da,\db, \dc>0$ be three 
strictly positive diffusion coefficients and let $\alpha,\beta,\gamma\ge1$. 

{
Consider non-negative initial data $(u_0,v_0,w_0)$ with finite mass and finite entropy $E(u_0,v_0,w_0)$ 
and assume that $(u,v,w)$ is 
a non-negative solution (be it either classical, weak or possibly even renormalised, see Remark \ref{regul}) 
of System \eqref{ABCa} -- \eqref{ABCc} on an time interval $0\le t<T$  for $T\le+\infty$, which {is assumed}  
\begin{itemize}
\item [i)] {to satisfy} the conservation laws \eqref{ABCCo} 
with two positive initial masses $M_1>0$, $M_2>0$ (for a.e. $0\le t<T$) and 
\item[ii)]
to dissipates the entropy \eqref{ABCEntr} (and thus equally the relative entropy \eqref{ABCRelEntr}) 
according to the entropy dissipation functional \eqref{ABCEntrDiss}, 
i.e. {satisfies the following entropy dissipation law
\begin{equation}\label{entroDiss}
E(u,v,w)(t_1) +\int_{t_0}^{t_1} D(u,v,w)(s)\le E(u,v,w)(t_0), \quad \text{for a.e.}\   0\le t_0 \le t_1<T,
\end{equation}
}which implies that the entropy $E(u,v,w)$ of the solution remains bounded and thus $(u(t), v(t), w(t)) \in ((L\log L)(\Omega))^3$ for a.e. $t\in[0,T)$. 
\end{itemize}}
\smallskip

Then, as long as such a solution $(u,v,w)$ exists {{(i.e.  for all  $t<T\le +\infty$)}} and  
satisfies i) and ii), it also satisfies the following  
exponential decay toward equilibrium:
\begin{multline}
\| u(t) - \ai\|_{L^1(\Omega)}  +  
\| v(t) - \bi\|_{L^1(\Omega)} + 
\| w(t) - \ci\|_{L^1(\Omega)} \\
 \le C(E(u_0,v_0,w_0) - 
E(\ai,\bi,\ci))\,e^{-K\,t}, \qquad \text{a.e.}\  0\le t<T,\label{res2}
\end{multline}
for a constant $C$ and a rate $K$, which can both be estimated explicitly in terms of the parameters
$\alpha, \beta, \gamma>1$, $\da,\db, \dc>0$, $M_1,M_2>0$ and the Poincar\'e- and the Logarithmic Sobolev constants 
$P(\Omega)$ and $L(\Omega)$ of the domain $\Omega$.
\end{The}

\begin{Rem}\label{regul}
{
The weak entropy dissipation law \eqref{entroDiss} certainly holds as an equality for all $0\le t_0 \le t_1 <T$ for any smooth solutions of suitably truncated approximations of Systems \eqref{ABCa} -- \eqref{ABCc}.
Thus, the weak entropy dissipation law \eqref{entroDiss} also holds as an equality for all classical and weak solutions, which satisfy $(u^{\alpha}v^{\beta} -w^{\gamma})\in L^p$ for a $p>1$, since the 
terms in the entropy dissipation are then uniformly integrable, which allows to pass to the limit. The existence of such weak solutions in $L^p$ with $p>1$ is a consequence of Proposition \ref{propun}, for instance, 
for all $N$ and $\alpha+\beta, \gamma \le 2$ (see \eqref{sdfg} and it's discussion in the Proof of Theorem \ref{ttexistence}) or for all $N$ and general $\alpha,\beta, \gamma$ provided sufficiently close diffusion coefficients $d_1, d_2, d_3$, which satisfy \eqref{eq:const-small-q}
for some $p'<2$.
For less regular solutions of \eqref{ABCa} -- \eqref{ABCc}, like weak $L^1$-solutions and renormalised solutions (see e.g. \cite{DFPV,Fis}),
assuming they allow to pass to the (a.e. pointwise) limit 
to obtain the weak entropy dissipation inequality \eqref{entroDiss} (e.g. by lower-semicontinuity, ...), Theorem \ref{tt2} still yields exponential convergence to equilibrium in $L^1$ \eqref{res2} (even if passing to the limit 
in a suitable truncated version of the entropy dissipation 
\eqref{ABCEntrDiss} might not be clear).  
}
\end{Rem}

\begin{Rem}
Note that {in cases, where e.g. classical solutions to 
system \eqref{ABCa} -- \eqref{ABCc} can be established,}      exponential convergence towards equilibrium in higher Lebesgue or Sobolev norms can be proven by an 
interpolation argument between the exponential decay of Theorem \ref{tt2} and polynomially growing
$H^k$ bounds (compare \cite{DF06}). 
Depending on the nonlinearities and the space dimension, 
Sobolev norms of any order may be created even if they do not initially exist
thanks to the smoothing properties of the heat kernel. 
\end{Rem}

The proof of Theorem \ref{tt2} is based on the so called entropy method, which aims to 
quantify the entropy dissipation \eqref{ABCEntrDiss} in terms of the relative entropy  \eqref{ABCRelEntr}.
More precisely, we shall prove a functional inequality, a so called entropy entropy-dissipation estimate (see Proposition \ref{ABCDDH} below) 
of the form 
\begin{equation*}
D(u,v,w) \ge K (E(u,v,w)-E(\ai,\bi,\ci)),
\end{equation*}
where $K$ is an explicitly computable positive constant and $(u,v,w)$ are non-negative functions, which satisfy the conservation laws \eqref{ABCCo}.
{
The convergence to equilibrium as stated in 
the proof of Theorem \ref{tt2} follows then from applying this entropy entropy-dissipation estimate point-wise in time 
as long as solutions dissipate entropy according to \eqref{entroDiss} 
(and obey the conservation laws \eqref{ABCCo}).
}

Previous results obtained by applying the entropy method to prove explicit exponential convergence to equilibrium 
for reaction diffusion systems were obtained in e.g \cite{DF06,DF08,DFIFIP} for systems featuring quadratic nonlinear reactions. 
In this paper, we present new ideas in proving an entropy entropy-dissipation estimate 
for 
{arbitrary (super)-linear, monomial reactions rates, which includes general nonlinear mass action law models for 
any reaction $\alpha\, \mathcal{U} + \beta\, \mathcal{V} \leftrightarrow \gamma\, \mathcal{W}$, see e.g. \cite{VVV}. }
\medskip

As second main result of the paper, we shall prove the following global existence Theorem, which partially 
extends the previously known existence results of the System 
\eqref{ABCa} -- \eqref{ABCc}, \eqref{ABCBC} and \eqref{ABCID} by applying duality estimates, which were recently 
developed in \cite{CDF}:

\begin{The} \label{ttexistence}
Let $\Omega$ be a sufficiently smooth  bounded ($\partial\Omega\in C^{2+\alpha}, \alpha>0$) and connected open set of 
$\R^N$ ($N\ge 1$). Let $\da,\db, \dc>0$ be three 
strictly positive diffusivity constants and $\alpha,\beta,\gamma\ge1$ and $(u_0,v_0,w_0)\in  (L^{\infty}(\Omega))^3$. Assume that $|d_1-d_3|$ (or equivalently $|d_2-d_3|$) is sufficiently small (to be specified below). 

Then, the System 
\eqref{ABCa} -- \eqref{ABCc}, \eqref{ABCBC} and \eqref{ABCID}
has a global classical solution.
\end{The}
%
\begin{Rem}
The result of Theorem \ref{ttexistence} seems to be new, for instance, in cases where $\alpha+\beta>\gamma$ are 
large and the diffusion coefficients are sufficiently close.  

However, up to our knowledge, the problem of global existence still holds open cases  when the diffusion coefficients are too different from each other and $2< \gamma\leq\alpha+\beta$.
\end{Rem}

We have organised this paper in the following manner. In Section 2, we first recall previously proven 
existence results and prove then Theorem \ref{ttexistence}, i.e. the
global existence of classical solutions for the System \eqref{ABCa} -- \eqref{ABCc} subject to non-negative bounded initial data $(u_0,v_0,w_0)$  and homogeneous Neumann boundary conditions \eqref{ABCBC} under the assumption that $d_1$ (or $d_2$) is sufficiently close to $d_3$.   

In Section 3, we establish an entropy entropy-dissipation estimate (Proposition \ref{ABCDDH}). 
Finally in Section \ref{Conv}, we prove a Cziszar-Kullback type inequality for the relative entropy functional 
\eqref{ABCRelEntr} and show, via a Gronwall argument,  
the exponential decay in $L^1$ towards the equilibrium $(\ai,\bi,\ci)$
as state in Theorem \ref{tt2}. 

%
%

\section{Existence Theory}
In the following, we denote 
$\Omega_T=(0,T)\times \Omega$ and for $p\in [1,+\infty)$
\begin{align*}
&\|u(t)\|_{L^p(\Omega)}=\Bigl(\int_\Omega |u(t,x)|^p\,dx\Bigr)^{1/p}, 
&&\|u\|_{L^p(\Omega_T)}=\Bigl(\int_0^T\!\!\int_\Omega |u(t,x)|^p\,dt dx\Bigr)^{1/p},\\
&\|u(t)\|_{L^\infty(\Omega)}=\esssup\limits_{x\in \Omega} |u(t,x)|, 
& &\|u\|_{L^\infty(\Omega_T)}=\esssup\limits_{(t,x)\in \Omega_T} |u(t,x)|.
\end{align*}
\noindent For the sake of completeness  and for the reader's convenience, let us recall the main known results on global existence of solutions to System
\eqref{ABCa} -- \eqref{ABCc} subject to the homogeneous Neumann boundary conditions \eqref{ABCBC} and non-negative initial data \eqref{ABCID}. At first, we shall provide  precise definitions of the notions of solutions see e.g. \cite{L2,P2}:
\medskip

By a {\it classical solution} to System \eqref{ABCa} -- \eqref{ABCc}, \eqref{ABCBC} and \eqref{ABCID}  on $\Omega_T=(0,T)\times \Omega$, we mean a triple of functions $(u,v,w)$ such that (at least)
\begin{itemize}
\item[i)] $(u,v,w)\in \mathcal{C}([0,T);L^1(\Omega)^3)\cap L^\infty([0,\tau]\times\Omega)^3$, $\forall \tau \in (0,T)$,
\item[ii)] $\forall p\in [1,+\infty)$ we have $u,v,w \in L^p(\Omega_T)$ and 
$\forall k,\ell=1\dots N$  
$$
\pa_t u,\pa_t v,\pa_t w \in L^p(\Omega_T), \quad \pa_{x_k}u,\pa_{x_k}v,\pa_{x_k}w\in L^p(\Omega_T), $$
$$
\pa_{x_kx_\ell}u, \pa_{x_kx_\ell}v,\pa_{x_kx_\ell}w \in L^p(\Omega_T),
$$
\item[iii)] the triple $(u,v,w)$ satisfy the equations  \eqref{ABCa} -- \eqref{ABCc} and the boundary conditions \eqref{ABCBC} a.e. (almost everywhere) on $\Omega_T$ and on $(0,T)\times\partial\Omega$, respectively in the sense of traces.
\end{itemize}
\medskip

By a {\it weak solution}  to System  \eqref{ABCa} -- \eqref{ABCc}, \eqref{ABCBC} and \eqref{ABCID}  on $\Omega_T$, we denote solutions essentially in the sense of distributions or, equivalently here, solutions in the sense of the variation of constants formula with the corresponding semigroups. More precisely, 
{we assume that  $w^\gamma-u^\alpha v^\beta\in L^1(\Omega_T)$ and} 
\begin{eqnarray*}
u(t)&=&S_{d_1}(t)u_0 + \alpha\int_0^t S_{d_1}(t-s)( w^\gamma(s)-u^\alpha(s) v^\beta(s))\,ds\\
v(t)&=&S_{d_2}(t)v_0 + \beta\int_0^t S_{d_2}(t-s)( w^\gamma(s)-u^\alpha(s) v^\beta(s))\,ds\\
w(t)&=&S_{d_3}(t)u_0 + \gamma\int_0^t S_{d_3}(t-s)( -w^\gamma(s)+u^\alpha(s) v^\beta(s))\,ds
\end{eqnarray*}
where $S_{d_i}(.)$ is the semigroup generated in $L^1(\Omega)$  by $-d_i\Delta$  with homogeneous Neumann boundary condition, $1\leq i\leq 3$.
\medskip

Provided non-negative initial data $u_0,v_0, w_0\in L^\infty(\Omega)$, the 
local-in-time existence and uniqueness of non-negative and uniformly bounded solution to (\ref{ABCa}) -- (\ref{ABCc}) are known (see e.g. \cite{R}). 
More precisely, there exists $T>0$ and a unique classical  solution $(u,v,w)$ of  \eqref{ABCa} -- \eqref{ABCc}, \eqref{ABCBC} and \eqref{ABCID} on $\Omega_T$.
If $T_{\max}$ denotes the maximal time of existence, then the solution 
triple $(u,v,w)$ ceases to remain bounded in the sense that
\begin{equation*}
\big(T_{\max}<+\infty\big)\Longrightarrow\lim_{t\nearrow T_{\max}}\left(\Vert u(t)\Vert_{L^\infty(\Omega)}+\Vert v(t)\Vert_{L^\infty(\Omega)}+ \Vert w(t)\Vert_{L^\infty(\Omega)}\right)=+\infty.
\end{equation*}
Moreover, it is well known that  in order to prove global-in-time existence 
(i.e. $T_{\max}=+\infty$), it is sufficient to obtain an a-priori estimate of the form
\begin{equation}\label{conti}
\forall t\in[0,T_{\max}),\qquad\Vert u(t)\Vert_{L^\infty(\Omega)}+\Vert v(t)\Vert_{L^\infty(\Omega)}+\Vert w(t)\Vert_{L^\infty(\Omega)}\leq H(t),
\end{equation}
where  $H : [0,+\infty)\rightarrow [0,+\infty)  $ is a non-decreasing and continuous function.
\medskip

However, an estimate like \eqref{conti} is far of being obvious for our System
except when the diffusion coefficients $d_1$, $d_2$, $d_3$ are equal, i.e  $d_1=d_2=d_3=d$. 
In this very special case, we have indeed that 
$Z=\beta\gamma u+\alpha\gamma v+2\alpha\beta w$ satisfies
$$
(E)
\begin{cases}
Z_t-d\Delta Z  =  0, &\qquad  (0,T_{\max})\times\Omega,\\
n(x)\cd\nx Z   =  0, &\qquad  (0,T_{\max})\times\partial\Omega,\\
Z(0,x)  =  Z_0(x), &\qquad x\in\Omega,\\
\end{cases}
$$
where $Z_0(x)=\beta\gamma u_0(x)+\alpha\gamma v_0(x)+2\alpha\beta w_0(x)$.
We may then deduce by the maximum principle for parabolic operators
(which doesn't hold for general parabolic systems) the following $L^{\infty}$ a-priori estimate 
$$ \forall t<T_{\max},\qquad
\|Z(t)\|_{L^\infty(\Omega)}\leq 
\|Z_0\|_{L^\infty(\Omega)}.
$$
Together with the non-negativity of the solution triple $(u,v,w)$, this implies
\begin{multline*}
\| u(t)\|_{L^\infty(\Omega)}+ \|v(t)\|_{L^\infty(\Omega)}+ \|w(t)\|_{L^\infty(\Omega)}
\leq 
\|Z_0\|_{L^\infty(\Omega)}\\
\leq \|\beta\gamma u_0+\alpha\gamma v_0+2\alpha\beta w_0\|_{L^\infty(\Omega)}, \qquad \forall t\leq T_{\max},
\end{multline*}
and $u(t)$, $v(t)$ and $w(t)$ are uniformly bounded in $L^\infty(\Omega)$ for all $t>0$ and therefore $T_{\max}=+\infty$.
\medskip

In general, the diffusion coefficients in System \eqref{ABCa}--\eqref{ABCc} are different from each other and the lack of a maximum principle makes the question 
of global existence considerably more complicated. 

Previously, the following cases have been studied by several authors:
\begin{enumerate}
\item Case  $\alpha=\beta=\gamma=1$:
 In this case, global existence of classical solutions was first obtained by Rothe \cite{R} for dimensions $N\leq 5$. Later, global classical solutions have first been proved by Pierre \cite{P1} for all dimensions $N$ and {then by Morgan \cite{M}, 
 Martin and Pierre \cite{MP} and Feng \cite{F}}. 
Exponential decay towards equilibrium has been shown by Desvillettes-Fellner \cite{DF06} for bounded solutions.
Global existence of weak solutions  has been proved by Laamri \cite{L} for initial data $u_0$,  $v_0$ and $w_0$ only in $L^1(\Omega)$. 

\item Second case $\gamma = 1$ and arbitrary $\alpha,\beta\ge 1$:
In this case, global existence of classical solutions has been obtained by Feng \cite{F} in  all dimensions $N$ and  more general boundary conditions. { It is also included in the results of \cite{M}}.

\item Third case  $\alpha+\beta\leq 2$ or $\gamma\leq 2$:
In this  case, Pierre \cite{P2} has proved global existence of weak solutions for initial data $u_0$,  $v_0$ and $w_0$ in $L^2(\Omega)$.
\item Fourth cases $\alpha+\beta< \gamma$ or ($1<\gamma < \frac{N+6}{N+2}$ and for arbitrary $\alpha,\beta\ge1$).
Laamri proved in \cite{L2}  global existence of classical solutions to System \eqref{ABCa}--\eqref{ABCc} , \eqref{ABCBC} and \eqref{ABCID} in the following cases:
\begin{enumerate}
\item $\alpha+\beta< \gamma$,
\item ($d_1=d_3$ or $d_2=d_3$) and for arbitrary $\alpha,\beta,\gamma\ge1$,
\item $d_1=d_2$ and for any  $(\alpha,\beta,\gamma)$ such that $\alpha+\beta\neq\gamma$,
\item $1<\gamma < \frac{N+6}{N+2}$ and for arbitrary $\alpha,\beta\ge1$.
\end{enumerate}
\end{enumerate}

Theorem \ref{ttexistence} below proves the existence of global solutions under a different set of assumptions and thus covers cases, which were open so far.
The proof is based on various duality estimates, in particular
on a recent improvement of certain duality estimates (see \cite{CDF}), which allow to obtain $L^p$ a-priori estimates ($p>2$) in cases where one of the diffusion coefficients $d_1$ or $d_2$ is sufficiently close to $d_3$. 
More precisely, we shall recall the following Proposition from \cite{CDF}:
\begin{Pro}[See Proposition 1.1 in \cite{CDF}]\label{propun} \hfill\\
  Let $\Omega$ be a bounded domain of $\R^N$ with smooth
  (e.g. $C^{2+\alpha}$, $\alpha>0$) boundary $\partial\Omega$, $T>0$,
  and $p\in (2,+\infty)$. We consider a coefficient function $M :=
  M(t,x)$ satisfying
  \begin{equation*}
    0 < a \leq M(t,x) \leq b < +\infty
    \quad \text{ for } (t,x) \in \Omega_T,
  \end{equation*}
  for some $0 < a < b < +\infty$, and an initial datum $u_0 \in
  L^p(\Omega)$.  \medskip

  Then, any weak solution $u$ of the parabolic system
  \begin{equation*}
    \left\{
      \begin{aligned}
        &\partial_t u - \Delta_x (M u) = 0 &&\quad \text{ on } \quad \Omega_T,
        \\
        &u(0,x) = u_0(x) &&\quad \text{ for } \quad x \in \Omega,
        \\
        &n(x)\cdot \nabla_x u  = 0 &&\quad \text{ on } \quad [0,T]
        \times \partial \Omega ,
      \end{aligned}
    \right.
  \end{equation*}
satisfies the estimate (where $p'<2$ denotes the H\"older conjugate
  exponent of $p$)
  \begin{equation}
    \label{eq:heat-estimate-forward}
    \|u\|_{L^p(\Omega_T)}
    \leq
    (1 + b \,D_{a,b,p'})\, T^{1/p} \,
    \|u_0\|_{L^p(\Omega)},
  \end{equation}
  and where for any $a,b>0$, $q\in (1,2)$
  \begin{equation*}
    D_{a,b,q}
    :=
    \frac{C_{\frac{a+b}2,q}}{1 - C_{\frac{a+b}2,q} \frac{b-a}{2}},
  \end{equation*}
  provided that the {following condition holds}
  \begin{equation}
    \label{eq:const-small-q}
    C_{\frac{a+b}2,p'} \, \frac{b-a}{2} < 1.
  \end{equation}
  Here, the constant $C_{m,q}>0$ is defined for $m>0$, $q \in (1,2)$
  as the best (that is, smallest) constant in the parabolic regularity
  estimate
  \begin{equation*}
    \| \Delta_x v\|_{L^q(\Omega_T)} \le C_{m,q}\,  \|f\|_{L^q(\Omega_T)},
  \end{equation*} 
  where  $v:[0,T] \times \Omega \to \R$ is the solution
  of the backward heat equation with homogeneous Neumann boundary conditions:
  \begin{equation*}
    \left\{
      \begin{aligned}
        &\partial_t v + m \,\Delta_x v = f &&\quad \text{ on } \quad \Omega_T,
        \\
        &v(T,x) = 0 &&\quad \text{ for } \quad x \in \Omega,
        \\
        &n(x)\cdot \nabla_x v  = 0 &&\quad \text{ on } \quad  [0,T]
        \times \partial \Omega.
      \end{aligned}
    \right.
  \end{equation*}

 We recall that one has $C_{m,q} < \infty$ for $m>0$, 
 {
 $q \in (1,2]$
 }and in particular $C_{m,2} \le\frac{1}{m}$.  Note that the constant $C_{m,q}$ may depend (besides on $m$
  and $q$) also on the domain $\Omega$ and the space dimension $N$,
  but {does not depend on the time $T$}.
\end{Pro}

Proposition \ref{propun} is the basis for the following proof of the existence Theorem \ref{ttexistence}:
%
\begin{proof}[\bf{Proof of Theorem \ref{ttexistence}}]\hfill\\ 
 The proof performs a bootstrap based on a-priori estimates derived on various duality estimates \cite{P2,CDF}.
 As a preliminary step, we shall assume without loss of generality that $|d_1-d_3|\le |d_2-d_3|$.
Otherwise, we rename $u\to v$, $d_1\to d_2$ and $\alpha\to\beta$ and vice versa.\\
First, we add the first equation \eqref{ABCa} and the third equation
\eqref{ABCc} to obtain 
\begin{equation*}
\pa_t (\gamma u+\alpha w) - \lap (M(t,x)(\gamma u+\alpha w))= 0, \quad \text{where}\quad M(t,x):=\frac{\da \gamma u + \dc \alpha w}{\gamma u+\alpha w},  
\end{equation*}
and observe that
\begin{equation*}
0< a := \min\{d_1,d_3\} \le M(t,x) \le b:=\max\{d_1,d_3\}<+\infty, \quad \forall (t,x)\in \Omega_T.
\end{equation*}
Thus, by applying the improved duality estimates derived in \cite{CDF}
and the non-negativity of the solutions, we have that
$u, w \in L^{p_0}(\Omega_T)$, where $p_0$ can be arbitrarily large provided 
that $b-a=|d_1-d_3|$ is 
assumed sufficiently small depending on the space dimension $N$, see \eqref{eq:const-small-q} above.
In particular, since $C_{m,2} \le\frac{1}{m}$ we estimate for \eqref{eq:const-small-q}
that
\begin{equation}\label{sdfg}
C_{\frac{a+b}2,2} \, \frac{b-a}{2} \le \frac{b-a}{a+b}< 1
\end{equation}
for all $0<a\le b<\infty$. 
{
Moreover, an interpolation argument between $C_{m,2}$ and $C_{m,3/2}$ shows that $C_{m,q}$ for $q\in[3/2,2]$ is bounded by a continuous function in $q$ and, as a consequence, that there exists always an exponent $p_0>2$ for any values of $d_1$ and $d_3$, see  \cite{CDF}.}
Finally, in all space dimensions, by assuming that $|d_1-d_3|$ is 
sufficiently small (depending on the space dimension $N$), we can find an exponent
\begin{equation*}
p_0=p_0(|d_1-d_3|)>\min\{\gamma,\alpha+\beta\},
\end{equation*}
such that at least one of the reaction terms on the right-hand sides of 
\eqref{ABCa}--\eqref{ABCc} is well-defined in $L^1$.

Next, we observe that 
\begin{equation*}
\beta(\pa_t u - \da\lap u)= \alpha(\pa_t v - \db\lap v )
\end{equation*}
and by applying another duality argument \cite[Lemma 3.4]{P2}, we obtain
\begin{equation*}
C_1 \| v \|_{L^p(\Omega_T)} \le \|u\|_{L^p(\Omega_T)} + 1 
\le C_2 \left(\| v \|_{L^p(\Omega_T)}+1\right), \qquad \forall p\in(1,+\infty).
\end{equation*}
Thus, we have that 
\begin{equation}\label{p0}
u, v, w \in L^{p_0}(\Omega_T), \qquad p_0=p_0(|d_1-d_3|)>\min\{\gamma,\alpha+\beta\}.
\end{equation}
In the following, we aim at bootstraping \eqref{p0}. We begin by observing that
\begin{equation*}
w^{\gamma} \in L^{\frac{p_0}{\gamma}}(\Omega_T), \qquad\text{and}\qquad 
u^\alpha v^\beta \in L^{\frac{p_0}{\alpha+\beta}}(\Omega_T),
\end{equation*}
where the later follows from Young's inequality, i.e. 
$$
u^\alpha v^\beta\le \frac{\alpha u^{\alpha+\beta} + \beta v^{\alpha+\beta}}{\alpha+\beta}.
$$
Then, parabolic regularity applied to the equations 
\eqref{ABCa}, \eqref{ABCb} implies that 
(see e.g. \cite{DF15}) 
\begin{equation}\label{q1a}
u, v \in L^{q_1}(\Omega_T), \qquad \frac{1}{q_1} > \frac{\gamma}{p_0} - \frac{2}{N+2}.
\end{equation}
Similarly, equation \eqref{ABCc} yields
\begin{equation}\label{q1b}
w \in L^{\tilde{q}_1}(\Omega_T), \qquad \frac{1}{\tilde{q}_1} > \frac{\alpha+\beta}{p_0} - \frac{2}{N+2}.
\end{equation}
By recalling that 
\begin{equation*}
\gamma(\pa_t u - \da\lap u)= \alpha(-\pa_t w + \dc\lap w), \;\;
\gamma(\pa_t v - \db\lap v) = \beta(-\pa_t w + \dc\lap w )
\end{equation*}
and by applying \cite[Lemma 3.4]{P2}, we have due to the first equality  
\begin{equation*}
C_1 \| w \|_{L^p(\Omega_T)} \le \|u\|_{L^p(\Omega_T)} + 1 
\le C_2 \left(\| w \|_{L^p(\Omega_T)}+1\right), \qquad \forall p\in(1,+\infty),
\end{equation*} 
and analogous
\begin{equation*}
C_1 \| w \|_{L^p(\Omega_T)} \le \|v\|_{L^p(\Omega_T)} + 1 
\le C_2 \left(\| w \|_{L^p(\Omega_T)}+1\right), \qquad \forall p\in(1,+\infty).
\end{equation*} 
We are thus able to conclude from \eqref{q1a} and \eqref{q1b} and the above norm equivalencies that
\begin{equation}\label{p1}
u, v, w \in L^{p_1}(\Omega_T), \qquad 
\frac{1}{p_1} > \frac{\min\{\gamma,\alpha+\beta\}}{p_0} - \frac{2}{N+2}.
\end{equation}
and we calculate that $p_1>p_0$ if and only if 
\begin{equation*}
p_0 > \frac{N+2}{2} \left( \min\{\gamma,\alpha+\beta\} - 1\right).
\end{equation*}
In this case, iterating the arguments between \eqref{p0} and \eqref{p1} constructs a sequence $p_{n+1} > p_n$, which leads after finitely many steps to an index $p_n$
such that
$$
\frac{\min\{\gamma,\alpha+\beta\}}{p_n} - \frac{2}{N+2} \le 0
$$
and a final bootstrap step yields $p_{n+1} = \infty$. Thus, the solutions $u,v,w$ are bounded in $L^{\infty}(\Omega_T)$ for any $T>0$ and the existence of global classical 
solutions follows by standard arguments.
\end{proof}

\textbf{{
\begin{Rem}
We also remark an alternative approach (with a somewhat different condition) to start the bootstrap in Theorem \ref{ttexistence}, 
which is based on the duality estimates of e.g. \cite{P2} and was kindly pointed out 
to us by a reviewer: By adding the first equation \eqref{ABCa} and the third equation
\eqref{ABCc}, we obtain 
\begin{equation*}
\pa_t (\gamma u+\alpha w) - d_1\lap (\gamma u+\alpha w)= \alpha (\dc-\da) \lap w. 
\end{equation*} 
Then, by the duality technique (see e.g. the proof of \cite[Lemma 3.4]{P2}) it follows directly for any $p\in (1,+\infty)$ and some  
$C=C(p,T)$ that 
$$ 
\|\gamma u+\alpha w\|_{L^p(\Omega_T)}\le C\left[1+|\da-\dc|\|w\|_{L^p(\Omega_T)}\right].
$$
Thus, with $\alpha\|w\|_{L^p(\Omega_T)}\le \|\gamma u+\alpha w\|_{L^p(\Omega_T)}$, the required estimate for $\|w\|_{L^p(\Omega_T)}$ holds provided that $|\da-\dc|$ is small enough. 
\end{Rem}}}

\section{Entropy entropy-dissipation estimate} \label{Convergence}

In this section, we prove Proposition \ref{ABCDDH}, which details an entropy entropy-dissipation estimate for $E(u,v,w)$, $D(u,v,w)$ defined in
\eqref{ABCEntr} and \eqref{ABCEntrDiss}. The proof uses the following technical (but essentially elementary) 
Lemmata \ref{ABCDConst} and \ref{ABCDDHCore}:

At first, Lemma \ref{ABCDConst} establishes a kind of entropy entropy-dissipation estimate 
in the special case of spatially-homogeneous non-negative concentrations
satisfying the conservation laws \eqref{ABCCo}. One could also say that Lemma \ref{ABCDConst} proves, what would be  
the key step of proving an entropy entropy-dissipation estimate for the 
homogeneous ODE system associated to \eqref{ABCa}--\eqref{ABCc}.

Secondly, Lemma \ref{ABCDDHCore} generalises the estimate of Lemma \ref{ABCDConst} to spatially-inhomogeneous concentrations.
%
%
\medskip

\noindent\textbf{Notations: } 
For notational convenience, we introduce capital letters as a short notation 
for square roots of lower case concentrations and overline for spatial averaging (remember that $|\Omega|=1$), i.e.
$$
U=\sqrt{u}\,, \qquad \Ai=\sqrt{\ai}\,,\qquad 
\Aq=\iO U(t,x)\,dx,
$$
and analog for $V$ and $W$.
{Finally, we denote by 
$
\|F\|_2^2 = \iO |F(x)|^2 \, dx$ 
the square of the $L^2(\Omega)$ norm for a given function $F: \Omega\to \R$.\\}
\medskip

We begin with: 
\begin{Lem} \label{ABCDConst} 
Let $\Ai,\;\Bi,\;\Ci$ denote the positive square roots of the steady state 
$\ai,\bi,\ci$. Let 
$a,b,c$ be non-negative constants satisfying the conservation laws 
\fref{ABCCo}, i.e. 
$$
\gamma a^2+\alpha c^2 = M_1 = \gamma\Ai^2+\alpha\Ci^2\qquad\text{and}\qquad 
\gamma b^2+\beta c^2 = M_2 = \gamma\Bi^2+\beta\Ci^2.
$$
%
%
Then, 
\begin{equation}
(a-\Ai)^2+(b-\Bi)^2+(c-\Ci)^2\le C\,(a^{\alpha}\,b^{\beta}-c^{\gamma})^2\,, \label{ABCDConstDH}
\end{equation}
where $C=C(\alpha,\beta,\gamma,M_1,M_2)$ is an explicitly computable positive constant.
\end{Lem}
\begin{proof}[\bf{Proof of Lemma \ref{ABCDConst}}]
The proof exploits a change of variables introducing perturbations $\ma,\mb,\mc$ around the equilibrium values 
$\Ai,\Bi,\Ci$, i.e.
\begin{equation}
a = \Ai(1+\ma),\qquad 
b = \Bi(1+\mb),\qquad 
c = \Ci(1+\mc). \label{ABCDConstan}
\end{equation}
Since the constants $a,b,c$ are non-negative, the new variable are bounded below, i.e. $\ma,\mb,\mc \ge-1$. 

Moreover, the conservation laws \fref{ABCCo} rewrite in the new variables into the relations 
\begin{align}
\gamma\Ai^2\, \ma (2+\ma) = -\alpha\Ci^2\,\mc(2+\mc), \label{ABCDConstABCDmu1} \\ 
\gamma\Bi^2\,\mb(2+\mb) = -\beta\Ci^2\,\mc(2+\mc). \label{ABCDConstABCDmu2} 
\end{align}
Since the mapping $\mu\mapsto \mu(2+\mu)$ is strictly monotone increasing on $[-1,\infty)$, 
it is straightforward to show that solving \eqref{ABCDConstABCDmu1} and \eqref{ABCDConstABCDmu2}
in terms of $\mc$ defines two strictly monotone decreasing functions $\mc \mapsto \ma(\mc)$ and 
$\mc \mapsto \mb(\mc)$, which have a unique zero crossing $\ma(0)=0=\mb(0)$ at $\mc=0$. More precisely,
\begin{eqnarray*}
&\ma(\mc) = \sqrt{1-\frac{\alpha\Ci^2}{\gamma\Ai^2}\mc(2+\mc)}-1,
\quad \ma \gtreqqless 0 \Leftrightarrow  \mc \lesseqqgtr 0,\qquad \mc\in[-1,\infty),\\
&\mb(\mc) = \sqrt{1-\frac{\beta\Ci^2}{\gamma\Bi^2}\mc(2+\mc)}-1,
\quad \mb \gtreqqless 0 \Leftrightarrow  \mc \lesseqqgtr 0,\qquad \mc\in[-1,\infty). \label{mbcdma}
\end{eqnarray*}

Moreover, the conservation laws \fref{ABCCo} or equivalently the relations \eqref{ABCDConstABCDmu1} and \eqref{ABCDConstABCDmu2} imply that the range of admissible $\mc$ is also bounded by above 
(since  $\ma(\mc)$, $\mb(\mc)$ are real and bounded below):    
\begin{equation*}
-1 \le \mc \le \mcu := -1+\sqrt{1+\min\left\{\frac{\gamma\Ai^2}{\alpha\Ci^2}\frac{\gamma\Bi^2}{\beta\Ci^2}\right\}}. 
\end{equation*}
Given these properties, we are able to express $\ma $ and $\mb$ in terms of $\mc$ in the following 
way:
\begin{eqnarray*}
&\ma = - R_1(\mc)\mc,\qquad   R_1(\mc):= \frac{\alpha\Ci^2}{\gamma\Ai^2}\frac{\mc+2}{\ma(\mc)+2},\\
&\mb = - R_2(\mc)\mc,\qquad   R_2(\mc):= \frac{\beta\Ci^2}{\gamma\Ai^2}\frac{\mc+2}{\mb(\mc)+2},
\end{eqnarray*}
and it is straightforward to check that the functions $R_1(\mc)$ and $R_2(\mc)$ are strictly monotone increasing and satisfy 
the following positive lower and upper bounds:
\begin{eqnarray*}
&0 < R_{1}(-1) < \frac{\alpha\Ci^2}{\gamma\Ai^2} = R_{1}(0) < R_{1}(\mcu) < \infty,\\
&0 < R_{2}(-1) < \frac{\beta\Ci^2}{\gamma\Ai^2} = R_{2}(0) < R_{2}(\mcu) < \infty.
\end{eqnarray*}
\smallskip

Next, we insert the ansatz \fref{ABCDConstan} into \fref{ABCDConstDH},
use the identity $\Ai^\alpha\Bi^\beta=\Ci^\gamma$ and observe that in order to prove Lemma \ref{ABCDConst}, we 
have to show that
\begin{equation}
\frac{\Ai^2\ma^2+\Bi^2\mb^2+\Ci^2\mc^2}
{\Ci^{2\gamma}\left((1+\mc)^\gamma-(1+\ma)^\alpha(1+\mb)^\beta\right)^2} \le C,
\label{ABCDConstDHa}
\end{equation}
for a constant $C$. We remark that for arbitrary perturbation $\ma, \mb, \mc \ge -1$ the fraction on the left-hand side of  \eqref{ABCDConstDHa} will not be bounded as there are many more zeros of the denominator than for the nominator, which vanishes only at $\ma=\mb=\mc=0$.

However, the conservations laws restrict the admissible perturbations $\ma=-R_1(\mc)\mc$ and 
$\mb=-R_2(\mc)\mc$ for $\mc\in[-1,\mcu]$.
Thus, we estimate the numerator of \fref{ABCDConstDHa}
\begin{equation} 
\Ai^2\ma^2+\Bi^2\mb^2+\Ci^2\mc^2
\le \mc^2\Ci^2\left(1+\frac{\Ai^2}{\Ci^2} R_1^2+\frac{\Bi^2}{\Ci^2}R_2^2
\right) 
\le C\,\mc^2,
\label{ABCDConstDHb}
\end{equation}
for a constant
 $C=C(\alpha,\beta,\gamma,M_1,M_2)$.

The denominator of \fref{ABCDConstDHa} can be estimated in the following way: We assume first that $\mc<0$, which is equivalent to $\ma(\mc)>0, \mb(\mc)>0$ by the 
monotonicity properties of $\ma(\mc)$ and $\mb(\mc)$.
Thus, we have $(1+\mc)^\gamma\le(1+\mc)$ and  $(1+\ma)^\alpha\ge(1+\ma)$ and $(1+\mb)^\beta\ge(1+\mb)$. Altogether in this case, we have 
\begin{align} \left|(1+\mc)^\gamma-(1+\ma)^\alpha(1+\mb)^\beta\right|\ge (1+\ma)(1+\mb)-(1+\mc)\nonumber \\
=\ma+\mb+\ma\mb-\mc\ge \ma-\mc = \left(R_1 +1\right)|\mc|\label{ABCDConstDHc}
\end{align}
since $\ma\ge-1$ implies $\mb+\ma\mb\ge0$. 

In the case $\mc\ge0$ and thus $\ma(\mc)\le0, \mb(\mc)\le0$, we have $(1+\mc)^\gamma\ge(1+\mc)$ and  $(1+\ma)^\alpha\le(1+\ma)$ and $(1+\mb)^\beta\le(1+\mb)$ and estimate in a similar way
\begin{align} \left|(1+\mc)^\gamma-(1+\ma)^\alpha(1+\mb)^\beta\right|\ge (1+\mc)-(1+\ma)(1+\mb)\nonumber \\
=\mc-\ma-\mb-\ma\mb\ge \mc-\ma = \left(1+R_1 \right)|\mc|\label{ABCDConstDHd}
\end{align}
since $\ma\ge-1$ implies $-\mb-\ma\mb\ge0$. 

Altogether, by \fref{ABCDConstDHc} and \fref{ABCDConstDHd}, 
we estimate the denominator of \fref{ABCDConstDHa} by
\begin{equation*}
\Ci^{2\gamma}\left((1+\mc)^\gamma-(1+\ma)^\alpha(1+\mb)^\beta\right)^2 \ge
\Ci^{2\gamma}\, \mc^2\,,
\end{equation*}
which proves with \fref{ABCDConstDHb} the statement of the Lemma \ref{ABCDConst}. 
\end{proof}
%

The following lemma extends Lemma \ref{ABCDConst} to non-negative 
functions $U,V,W$ whose squares $U^2,V^2,W^2$ satisfy the conservation laws \fref{ABCCo}:
%
%
\begin{Lem} \label{ABCDDHCore} 
Let $(\Ai,\,\Bi,\,\Ci)$ denotes the positive square roots of the steady state $(\ai,\bi,\ci)$
and $U,\,V,\,W : \Omega \rightarrow \R$ be measurable, non-negative functions 
satisfying the conservation laws 
\fref{ABCCo}, i.e. 
$$
\gamma\ov{U^2}+\alpha\ov{W^2} = M_1 = \gamma U_{\infty}^2+\alpha W_{\infty}^2,\qquad \gamma\ov{V^2}+\beta\ov{W^2} = M_2 = \gamma V_{\infty}^2+\beta W_{\infty}^2.
$$

Then, the following estimate holds for any $\alpha, \beta, \gamma \ge 1$
\begin{multline}
\|U-\Ai\|_2^2+\|V-\Bi\|_2^2+\|W-\Ci\|_2^2 \le K_1\|W^\gamma-U^\alpha V^\beta\|_2^2 \\
+ K_2\left(\|U - \Aq\|_2^2+\|V - \Bq\|_2^2+\|W - \Cq\|_2^2\right), 
\label{ABCDDHone}
\end{multline}
for various constants $K_1$ and $K_2$ depending only on $\alpha,\beta,\gamma\ge1$, $d_1,d_2,d_3>0$ and 
$M_1,M_2>0$.
\end{Lem}
%
%
\begin{proof}[\bf{Proof of Lemma \ref{ABCDDHCore}}]\hfill \\
\noindent\textbf{Step 1 : }  In order to prove \eqref{ABCDDHone}, we shall first show that 
\begin{multline}
\|W^\gamma-U^\alpha V^\beta\|_2^2 \ge 
\frac{1}{2}\left(\Cq^\gamma-\Aq^\alpha \Bq^\beta\right)^2\\- 
C\left(\|U-\Aq\|_2^2+\|V-\Bq\|_2^2+\|W-\Cq\|_2^2\right),\label{ABCnonlin}
\end{multline}
for some constants $C$.
We remark that due to Jensen's inequality and the conservations laws \eqref{ABCCo}, we have the following natural bounds for the averages 
$$
\Aq\le\sqrt{\Aqq}\le C_M,\qquad 
\Bq\le\sqrt{\Bqq}\le C_M,\qquad 
\Cq\le\sqrt{\Cqq}\le C_M,
$$ 
{for a constants $C_M=C_M(\alpha,\beta,\gamma,M_1,M_2)$.}
Thus, also the reaction rate
$$
|\Aq^{\alpha}\Bq^{\beta}-\Cq^{\gamma}|\le C_M,
$$ 
is bounded in terms of the (positive) initial masses $M_1,M_2$ by various constants 
$C_M=C_M(\alpha,\beta,\gamma,M_1,M_2)$.

{
Moreover, we introduce the following deviations  
with zero mean value:
\begin{equation}
\dela(x) = U - \Aq, \quad \delb(x) = V - \Bq,\quad 
\delc(x)=W - \Cq, \qquad\ov{\dela}=\ov{\delb}=\ov{\delc}=0.
\label{ABCDan1}
\end{equation}}
\medskip

\noindent\underline{Step 1a:} We consider for a constant $K> 0$ the set 
$$
S:= \left\{x\in\Omega \ | \ |\dela| \le K,\quad |\delb| \le K,\quad 
|\delc| \le K\right\}.
$$

The set $S$ constitutes the part of the state space $(U,V,W)=(\Aq+\dela,\Bq+\delb,\Cq+\delc$) which includes the states close to equilibrium at which 
$(\Aq,\Bq,\Cq)=(\Ai,\Bi,\Ci)$ and $(\dela,\delb,\delc)=0$ holds. Using the bounds $-\Aq \le \dela\le K$, $-\Bq \le \delb\le K$ and $-\Cq \le \delc\le K$ (where to lower bounds are always satisfied due to definition \eqref{ABCDan1}), we have 
with $\alpha, \beta, \gamma \ge 1$
\begin{align*}
U^{\alpha}V^{\beta}&= (\Aq+\dela)^{\alpha}(\Bq+\delb)^{\beta} = \Aq^{\alpha}\Bq^{\beta} + T_1(x)(\dela+\delb),\\
W^{\gamma}&= (\Cq+\delc)^{\gamma} = \Cq^{\gamma} + T_2(x)\delc,
\end{align*}
for bounded remainder terms 
{
$T_1$ and $T_2$, which satisfy $|T_1(x)|\le C(K^{\alpha+\beta-1})$ and $|T_2(x)|\le C(K^{\gamma-1})$, respectively, for all $x\in S$.}
Thus, we can factorise {and estimate with Young's inequality}
\begin{align}
\|U^{\alpha}V^{\beta}-W^{\gamma}\|_{L^2(S)}^2 &=  \|\Aq^{\alpha}\Bq^{\beta}-\Cq^{\gamma}\|_{L^2(S)}^2\nonumber\\
&\quad +2\int_{S}(\Aq^{\alpha}\Bq^{\beta}-\Cq^{\gamma})
\Bigl(T_1(\dela+\delb)-T_2\delc\Bigr)\,dx\nonumber\\
&\quad +\|T_1(\dela+\delb)-T_2\delc\|_{L^2(S)}^2\nonumber\\
&{\ge \frac{1}{2}\|\Aq^{\alpha}\Bq^{\beta}-\Cq^{\gamma}\|_{L^2(S)}^2
-\|T_1(\dela+\delb)-T_2\delc\|_{L^2(S)}^2}\nonumber\\
&\ge\frac{1}{2}\|\Aq^{\alpha}\Bq^{\beta}-\Cq^{\gamma}\|_{L^2(S)}^2
-C(K,M_1,M_2)(\ov{\dela^2}+\ov{\delb^2}+\ov{\delc^2}),
\label{ABCDDHonec}
\end{align}
{where we have used that $|T_1|\le C(K)$, $|T_2|\le C(K)$.}
\medskip

\noindent\underline{Step 1b:} It remains to consider the orthogonal set of all points 
$$
S^\perp:= \left\{x\in\Omega \ |\  |\dela| > K,\quad\text{or}\quad |\delb| > K,\quad\text{or}\quad |\delc| > K\right\}.
$$
By Chebyshev's inequality it follows that
$$
\left| \{\dela^2 > K^2\}\right| \le \frac{\ov{\dela^2}}{K^2}, \qquad
\left| \{\delb^2 > K^2\}\right| \le \frac{\ov{\delb^2}}{K^2}, \qquad
\left| \{\delc^2 > K^2\}\right| \le \frac{\ov{\delc^2}}{K^2}.
$$
Moreover, we observe that $|\{|\dela| > K\}|=|\{\dela^2 > K^2\}|$. 
This implies that $|S^\perp|\le C(K)(\ov{\dela^2}+\ov{\delb^2}+\ov{\delc^2})$.
Therefore, since $|\Aq^{\alpha}\Bq^{\beta}-\Cq^{\gamma}|\le C_M$ is bounded, we have 
$$
\|\Aq^{\alpha}\Bq^{\beta}-\Cq^{\gamma}\|_{L^2(S^\perp)}^2
\le  C(K,M_1,M_2)(\ov{\dela^2}+\ov{\delb^2}+\ov{\delc^2})
$$
and
\begin{multline}
\|U^{\alpha}V^{\beta}-W^{\gamma}\|_{L^2(S^\perp)}^2 \ge0\\
\ge  \|\Aq^{\alpha}\Bq^{\beta}-\Cq^{\gamma}\|_{L^2(S^\perp)}^2
-C(K,M_1,M_2)(\ov{\dela^2}+\ov{\delb^2}+\ov{\delc^2}).
\label{ABCDDHoned}
\end{multline}
Altogether, the estimates \eqref{ABCDDHonec} and \eqref{ABCDDHoned} 
finish the proof of the estimate \eqref{ABCnonlin} and conclude Step 1.
\medskip

We remark that estimate \eqref{ABCDDHoned} could be restricted to the subset 
$$
S^\perp_{\mu}:= \{x\in S^\perp \ | \ |U^{\alpha}V^{\beta}-W^{\gamma}|\le \mu\},
$$
for any $\mu>0$ since the estimate \eqref{ABCnonlin} holds naturally at all points 
$x\in\Omega$ where $|W^\gamma-U^\alpha V^\beta|>\mu$ holds, i.e. 
$$
\|U^{\alpha}V^{\beta}-W^{\gamma}\|_{L^2(S^\perp_{\mu})}^2
\ge \min\left\{\frac{\mu^2}{|\Aq^{\alpha}\Bq^{\beta}-\Cq^{\gamma}|^2},1\right\}
\|\Aq^{\alpha}\Bq^{\beta}-\Cq^{\gamma}\|_{L^2(S^\perp_{\mu})}^2.
$$
\medskip

\noindent\textbf{Step 2 :}
Inserting estimate \eqref{ABCnonlin} into \eqref{ABCDDHone} allows now to apply Lemma \ref{ABCDConst} 
after expanding the left hand side of \eqref{ABCDDHone} in terms of 
\begin{equation*}
\dela(x) = U - \Aq, \quad \delb(x) = V - \Bq,\quad 
\delc(x)=W - \Cq, 
\end{equation*}
Moreover, we shall adapt the ansatz used in Lemma \ref{ABCDConst} in the following fashion
\begin{equation*}
\Aqq = \Ai^2(1+\ma)^2,\qquad 
\Bqq = \Bi^2(1+\mb)^2,\qquad 
\Cqq = \Ci^2(1+\mc)^2,
\label{ABCDan2}
\end{equation*}
and $-1\le\ma,\mb,\mc$, which recovers the relations \eqref{ABCDConstABCDmu1} and \eqref{ABCDConstABCDmu2} as well as the subsequent study of $\ma(\mc)=-R_1(\mc)\mc$ and $\mb(\mc)=-R_2(\mc)\mc$ as functions of $\mc\in[-1,\mcu]$ (see Lemma \ref{ABCDConst}).
\smallskip

The definitions of \fref{ABCDan1} imply readily for the  right-hand side of \fref{ABCDDHone} that
\begin{equation}
\|U-\Aq\|_2^2=\Aqq-\Aq^2=\ov{\dela^2}=\|\dela\|_2^2,\label{ABCDDHoner2}
\end{equation} 
and analog $\ov{\delb^2}=\Bqq-\Bq^2$ and $\ov{\delc^2}=\Cqq-\Cq^2$. Thus, 
$\frac{\ov{\dela^2}}{\sqrt{\Aqq}+\Aq} = \sqrt{\Aqq}-\Aq$
and it follows that 
\begin{eqnarray}
\Aq = \Ai(1+\ma)-\frac{1}{\sqrt{\Aqq}+\Aq}\,\ov{\dela^2},
\label{ABCDDHoner1} 
\end{eqnarray}
and analog expressions hold for $\Bq$ and $\Cq$.
Therefore, the first terms on the left-hand side of \fref{ABCDDHone} expands with \fref{ABCDDHoner1} to 
\begin{eqnarray}
\|U-\Ai\|_2^2 = \Ai^2\ma^2+\frac{2\Ai}{\sqrt{\Aqq}+\Aq}\,\ov{\dela^2},
\label{ABCDDHonel} 
\end{eqnarray} 
and analog for $\|V-\Bi\|_2^2$ and $\|W-\Ci\|_2^2$.
Note that \eqref{ABCDDHonel} is quadratic in the perturbation $\ma$ and the deviation $\dela$.
\smallskip

We point out that in \eqref{ABCDDHonel} the expansions in terms of $\ov{\dela^2}$ is unbounded for vanishing $\Aqq\ge \ov{U}^2$. Although such states with $\Aqq\sim0$ or $\Bqq\sim0$ or $\Cqq\sim0$ (in which the positive initial mass is concentrated in the remaining non-vanishing concentrations) are highly degenerate states far from equilibrium, we are nevertheless forced to distinguish two cases.
\medskip

 Firstly we consider the \medskip\\
\underline{Case $\Aqq\ge\eps^2$, $\Bqq\ge\eps^2$, $\Cqq\ge\eps^2$:} for a constant $\eps>0$ to be chosen later in the opposite cases for small $\Aqq$, $\Bqq$, $\Cqq$.
In this case, the expansion \fref{ABCDDHonel} is not degenerated 
and {
we have 
$
\frac{1}{\sqrt{\Aqq}+\Aq},\frac{1}{\sqrt{\Cqq}+\Cq},\frac{1}{\sqrt{\Cqq}+\Cq}
\le C(\eps)
$
and 
\begin{equation}\label{dfgdfg}
\frac{2\Ai}{\sqrt{\Aqq}+\Aq},\frac{2\Ci}{\sqrt{\Cqq}+\Cq},\frac{2\Ci}{\sqrt{\Cqq}+\Cq}\le C(\eps,\alpha,\beta,\gamma,M_1,M_2).
\end{equation}
Moreover, we estimate first by using \eqref{ABCDDHoner1} that
\begin{multline}\label{sdfsdf}
\|\Aq^{\alpha}\Bq^{\beta}-\Cq^{\gamma}\|_2^2 \ge
\Ci^{2\gamma}\left[(1+\ma)^{\alpha}(1+\mb)^{\beta}-
(1+\mb)^{\gamma}\right]^2\\
-C(\eps,\alpha,\beta,\gamma,M_1,M_2)\left(\ov{\dela^2}+\ov{\delb^2}+\ov{\delc^2}\right).
\end{multline}
Next, we insert \fref{ABCDDHonel} into the left-hand side  of \fref{ABCDDHone} (and use \eqref{dfgdfg}) and insert \eqref{ABCnonlin}, \fref{ABCDDHoner2} and 
\eqref{sdfsdf} into the  right-hand side  of \fref{ABCDDHone}. Thus, 
in order to finish the proof of the Lemma, it is sufficient to show that }
\begin{multline}\label{nondeg}
\Ai^2\ma^2+\Bi^2\mb^2+\Ci^2\mc^2 \le 
K_1{\Ci^{2\gamma}\left[(1+\ma)^{\alpha}(1+\mb)^{\beta}-
(1+\mb)^{\gamma}\right]^2}\\ 
+ \left(K_2-C(\eps,\alpha,\beta,\gamma,K,K_1,M_1,M_2) \right)
\left(\ov{\dela^2}+\ov{\delb^2}+\ov{\delc^2}\right),
\end{multline}

Now, provided that we chose $K_2\ge C(\eps,\alpha,\beta,\gamma,K,K_1,M_1,M_2)$ sufficiently large, the inequality \eqref{nondeg} follows directly from Lemma \ref{ABCDConst}, 
in particular {from the reformulated estimate 
\eqref{ABCDConstDHa}, which is sufficient to prove estimate 
\eqref{ABCDConstDH}.}
\medskip

Secondly we  treat the \medskip\\
\underline{Case $\Aqq\le\eps^2$ or $\Bqq\le\eps^2$ or $\Cqq\le\eps^2$:} In these degenerate cases, the expansion \eqref{ABCDDHoner1} is not useful.
However, since the corresponding states $(U,V,W)$ are far from equilibrium, the inequality in \fref{ABCDDHone} can be established by direct estimates (i.e. without Lemma \ref{ABCDConst}).

First, we observe that the {left}-hand side  of \eqref{ABCDDHone} is bounded in terms of the conservation laws \eqref{ABCCo}
\begin{multline}
\|U-\Ai\|_2^2+\|V-\Bi\|_2^2+\|W-\Ci\|_2^2 \le
\\\iO (u+v+w + \ai+\bi+\ci)\,dx
\le C_M(M_1,M_2,\alpha,\beta,\gamma). \label{ABCDDHonf}
\end{multline}
Next, we consider the case that $\Aq\le\eps^2$, where $\eps$ is a constant to be specified below. Rewriting the conservation law $\gamma\Aqq+\alpha\Cqq=M_1$ using $\Aqq-\Aq^2=\ov{\dela^2}$ and $\Cqq-\Cq^2=\ov{\delc^2}$, we estimate with $\Aq\le\eps^2$ that
$$
\Cq^2\ge \frac{M_1}{\alpha}-\ov{\delc^2}
-\frac{\gamma}{\alpha}\left(\ov{\dela^2}+\eps^2\right).
$$
Thus, since $M_1>0$
\begin{multline*}
\left(\Cq^\gamma-\Aq^\alpha \Bq^\beta\right)^2
\ge  \Cq^{2\gamma}-2\Cq^\gamma\Aq^\alpha\Bq^\beta\\
\ge \left(\frac{M_1}{\alpha}-\ov{\delc^2}
-\frac{\gamma}{\alpha}\ov{\dela^2}-\frac{\gamma}{\alpha}\eps^2\right)^{\gamma}-2\Cq^\gamma\Bq^\beta\eps^{\alpha}\ge
C>0,
\end{multline*}
for a strictly positive constant $C>0$
provided that $\ov{\dela^2}+\ov{\delc^2}$ is sufficiently small and $\eps$ is subsequently chosen small enough. 

Therefore, there exist two constants $K_1$ and $K_2$ such that
\begin{multline}
\|U-\Ai\|_2^2+\|V-\Bi\|_2^2+\|W-\Ci\|_2^2 \le
C_M \\
\le K_1 \left(\Cq^\gamma-\Aq^\alpha \Bq^\beta\right)^2
+ K_2 \left(\ov{\dela^2}+\ov{\delb^2}+\ov{\delc^2}\right).
\label{ABCDHong}
\end{multline}
In the opposite case that $\ov{\dela^2}+\ov{\delc^2}$ is not sufficiently small \eqref{ABCDHong} still holds for a sufficiently large constant $K_2$.

The case that $\Bq\le\eps^2$ can be treat analog to  the case $\Aq\le\eps^2$ above.
\medskip

Finally, in the case that $\Cq\le\eps^2$, we estimate the conservation laws $\gamma\Aqq+\alpha\Cqq=M_1$ and $\gamma\Bqq+\beta\Cqq=M_2$ like above as
$$
\Aq^2\ge \frac{M_1}{\gamma}-\ov{\dela^2}
-\frac{\alpha}{\gamma}\left(\ov{\delc^2}+\eps^2\right), \qquad
\Bq^2\ge \frac{M_2}{\gamma}-\ov{\delb^2}
-\frac{\beta}{\gamma}\left(\ov{\delc^2}+\eps^2\right),
$$
which yields 
\begin{align*}
\left(\Cq^\gamma-\Aq^\alpha \Bq^\beta\right)^2
\ge& 
\biggl(\frac{M_1}{\gamma}-\ov{\dela^2}-\frac{\alpha}{\gamma}\left(\ov{\delc^2}+\eps^2\right)\biggr)^{\alpha}
\biggl(\frac{M_2}{\gamma}-\ov{\delb^2}-\frac{\beta}{\gamma}\left(\ov{\delc^2}+\eps^2\right)\biggr)^{\beta}\\
&-2\Aq^\alpha\Bq^\beta\eps^{\gamma}\ge C>0,
\end{align*}
for a strictly positive constant $C>0$
provided that $\ov{\dela^2}$,    $\ov{\delc^2}$, $\ov{\delc^2}$ are sufficiently small and $\eps$ is subsequently chosen small enough. 

Similar to above, this yields \eqref{ABCDHong}
and together with \eqref{ABCnonlin} finishes the proof of the Lemma \ref{ABCDDHCore}. 
\end{proof}
\bigskip

Next, we recall the Logarithmic Sobolev inequality, which holds 
on a bounded domain (without confining potential) as a consequence of the inequalities of Sobolev and Poincar\'e:

\begin{Lem}[Logarithmic Sobolev inequality] \label{Lemma:LogSob}
 Let $\phi :\Omega\rightarrow \R$ where $\Omega$ is a bounded domain in $\R^{N}$ such that the Poincar\'e (-Wirtinger)
 and Sobolev inequalities hold:
\begin{eqnarray*}
& \|\phi - \frac{1}{|\Omega|}\iO \phi \, dx\|_{L^2(\Omega)}^2 \le P(\Omega) \, \|\nx \phi\|_{L^2(\Omega)}^2, \\
&\|\phi\|_{L^q(\Omega)}^2 \le C_1(\Omega) \,
\|\nx \phi\|_{L^2(\Omega)}^2 + C_2(\Omega)\,
 \|\phi\|_{L^2(\Omega)}^2\,, 
\qquad \frac{1}{q} = \frac{1}{2}-\frac{1}{N},
\end{eqnarray*}
Then, the logarithmic Sobolev inequality 
\begin{equation}\label{LogSob}
 \iO \phi^2 \ln\left(\frac{\phi^2}{\|\phi\|_{L^2(\Omega)}^2}\right) dx \le L(\Omega,N)
\, \|\nx \phi\|_{L^2(\Omega)}^2
\end{equation}
holds for some constant $L(\Omega,N)>0$ in any space dimension $N$.
\end{Lem}
The proof follows an argument of Strook \cite{Logsob}, see \cite{DFIFIP}.
\vskip3mm
We are now in position to state the entropy entropy-dissipation estimate 
for the relative entropy $E-E_{\infty}$ defined in \eqref{ABCEntr} and the entropy dissipation $D$ \eqref{ABCEntrDiss}, which holds for admissible 
functions regardless if or if not they are solutions (at a given time $t$) of the System \eqref{ABCa} -- \eqref{ABCc}: 
\begin{Pro} \label{ABCDDH}
{Let $\Omega$ be a connected open set of $\R^N$ $(N\ge 1)$ with
sufficiently smooth boundary $\partial\Omega$  such that Poincar\'e's inequality and the Logarithmic Sobolev inequality hold.}

Let $u$, $v$, $w$ be measurable non-negative functions from $\Omega$ into $\R^+$ 
such that 
$\iO (\gamma u + \alpha w) = M_1$ and $\iO (\gamma v + \beta w) = M_2$. Then,
\begin{equation*}
D(u,v,w) \ge K (E(u,v,w)-E(\ai,\bi,\ci)). 
\end{equation*}
for a constant $K=K(d_1,d_2,d_3,\alpha,\beta,\gamma,M_1,M_2,P(\Omega),L(\Omega))$ depending only on the positive diffusion coefficients $d_1,d_2,d_3>0$, the stoichiometric coefficients $\alpha,\beta,\gamma\ge1$, the positive initial masses $M_1,M_2>0$ and the Poincar\'e and Logarithmic-Sobolev constants $P(\Omega)$ and $L(\Omega)$ of the domain $\Omega$.
\end{Pro}
%
\begin{proof}[\bf{Proof of Proposition \ref{ABCDDH}}]\hfill\\ 
We begin by rewriting the relative entropy \eqref{ABCRelEntr} using the following additivity property
\begin{align}
E(t) - E_{\infty} 
=&\iO \left( u\ln\Bigl(\frac{u}{\aq}\Bigr)+v\ln\Bigl(\frac{v}{\bq}\Bigr)
+w \ln\Bigl(\frac{w}{\cq}\Bigr) \right)dx\label{EEDa}\\
&+\aq\ln\Bigl(\frac{\aq}{\ai}\Bigr)-(\aq-\ai)+\bq\ln\Bigl(\frac{\bq}{\bi}\Bigr)-(\bq-\bi) \label{EEDb}\\
&+\cq \ln\Bigl(\frac{\cq}{\ci}\Bigr)-(\cq-\ci),\nonumber
\end{align}
where the integral \eqref{EEDa} corresponds to the relative entropy of 
concentrations ($u,v,w$) with respect to their averages ($\aq,\bq,\cq$) 
and \eqref{EEDb} is the relative entropy of the 
averages ($\aq,\bq,\cq$) with respect to the equilibrium ($\ai,\bi,\ci$). 

The integral \eqref{EEDa} can be estimated thanks to
the Logarithmic Sobolev inequality \eqref{LogSob} 
by 
\begin{equation} 
\iO u\ln\Bigl(\frac{u}{\aq}\Bigr)\,dx 
\le L(\Omega) \iO \frac{|\nx u|^2}{u} dx= 4L(\Omega) \iO \left|\nx U\right|^2 dx,
\label{EEDc}
\end{equation}
and similar for $v$ and $w$. Thus,  \eqref{EEDa} is bounded above by the Fisher information part of entropy dissipation  \eqref{ABCEntrDiss}.

On the other hand, the relative entropy \eqref{EEDb} can be estimated in the following way:
We define the function
$$\phi(x,y) = \frac{x\,\ln(x/y) - (x-y)}{(\sqrt{x}-\sqrt{y})^2}=\phi({x}/{y},1), 
$$
which extends continuously to $[0,+\infty)\times(0,+\infty)\rightarrow \R_+$
and note that thanks to the conservation laws \eqref{ABCCo}
the expressions 
$$
\phi(\aq/\ai,1)\le C_M,\qquad\phi(\bq/\bi,1)\le C_M,\qquad\phi(\cq/\ci,1)\le C_M,
$$
are bounded by a constant $C_M(M_1,M_2)$ in terms of the initial masses. Thus, 
{by using as in Step 2 of Lemma \ref{ABCDDHCore}
that $ \sqrt{\aq}= \sqrt{\Aqq} = \Ai(1+\ma)$, we have }
\begin{equation} 
\aq\ln\Bigl(\frac{\aq}{\ai}\Bigr)-(\aq-\ai)\le C_M
(\sqrt{\aq}-\sqrt{\ai})^2\le C_M \Ai^2 \,\ma^2,
\label{EEDd}
\end{equation}
and similar estimates hold for $\bq$ and $\cq$.
\medskip

Next, we estimate the entropy dissipation \eqref{ABCEntrDiss} below by using the elementary inequality $(u^{\alpha}v^{\beta}-w^{\gamma})(\ln(u^{\alpha}v^{\beta})-\ln(w^{\gamma}))\ge 
4(U^{\alpha}V^{\beta}-W^{\gamma})^2$, Poincar\'e's inequality and the Logarithmic Sobolev inequality and obtain the estimate 
\begin{align}
D(u,v,w) \ge&\ 4\left\|U^{\alpha}V^{\beta}-W^{\gamma}\right\|^2_2\nonumber \\
&+ \theta C_P\left(\left\|U-\Aq\right\|^2_2
+ \left\|V-\Bq\right\|^2_2
+ \left\|W-\Cq\right\|^2_2\right)\nonumber \\
&+ 4(1-\theta)\Bigl(\iO \left|\nx U\right|^2 dx
+ \iO \left|\nx V\right|^2 dx
+ \iO \left|\nx W\right|^2 dx\Bigr),
\label{EEDf}
\end{align}
for a constant $C_P=C_P(\da,\db,\dc,P(\Omega))$ 
with the Poincar\'e constant $P(\Omega)$ and a constant $\theta\in(0,1)$ to be chosen {such that the last term on the  right-hand side  of \eqref{EEDf} controls via the Logarithmic Sobolev inequality \eqref{EEDc} the first contribution to the relative entropy $E-E_{\infty}$, i.e. the integral \eqref{EEDa}. 
}

Combining the expressions of the relative entropy \eqref{ABCRelEntr} and the entropy dissipation \eqref{ABCEntrDiss}
with the estimates \eqref{EEDd} and \eqref{EEDf} and choosing $\theta$ in \eqref{EEDf} {in order to control \eqref{EEDa} with the last term of  \eqref{EEDf}}, it remains to show that
\begin{multline*}
C_M\left(\Ai^2 \ma^2+\Bi^2 \mb^2+\Ci^2 \mc^2\right) \le 
K_1 \left\|U^{\alpha}V^{\beta}-W^{\gamma}\right\|^2_2 \\
+ K_2\left(\left\|U-\Aq\right\|^2_2
+ \left\|V-\Bq\right\|^2_2
+ \left\|W-\Cq\right\|^2_2\right),
\end{multline*}
which is a consequence of Lemma \ref{ABCDDHCore} 
{after observing that 
$$
\|U-\Ai\|_2^2 \le \|U - \Aq\|_2^2 + \|\Aq-\Ai\|_2^2 = \|U - \Aq\|_2^2 + \Ai^2\ma^2.
$$
and analog for $\|V-\Bi\|_2^2$ and $\|W-\Ci\|_2^2$.}
This ends the proof of Proposition \ref{ABCDDH}.
\end{proof}

\section{Estimates of convergence towards equilibrium} \label{Conv}

In this section, we use the estimates of Section \ref{Convergence} in order 
to prove Theorem \ref{tt2}.
We begin with a Cziszar-Kullback type inequality relating
convergence in relative entropy to convergence in $L^1$.

\begin{Pro} \label{ABCDCK}
For all (measurable) functions $u,v,w : \Omega \to (\R_+)^3$, 
for which $\iO (\gamma u + \alpha w) = M_{1}>0$ and $\iO (\gamma v + \beta w) = M_{2}>0$ holds, we have the following Cziszar-Kullback type inequality
for the entropy functional $E(u,v,w)$ defined in \fref{ABCEntr}:
\begin{multline*} 
E(u,v,w) - E(\ai,\bi,\ci) \\
\ge C\left(\|u-\ai\|_{L^1(\Omega)}^2+\|v-\bi\|_{L^1(\Omega)}^2+\|w-\ci\|_{L^1(\Omega)}^2\right),
\end{multline*}
for a constant $C=C(M_1,M_2,\alpha,\beta,\gamma)>0$ 
depending on the masses $M_1,M_2>0$ the stoichiometric coefficients 
$\alpha,\beta,\gamma\ge1$.
\end{Pro}
\begin{proof}[\bf{Proof of Proposition \ref{ABCDCK}}]\hfill\\
For a constant $x_\infty>0$, we define (the continuous extension of) the non-negative relative entropy density function 
$$
0\le q(x,x_{\infty}):=x\ln \Bigl(\frac{x}{x_{\infty}}\Bigr) - (x-x_{\infty}), \qquad x\in[0,+\infty),
$$ 
and rewrite the relative entropy \eqref{ABCRelEntr} as (see also \eqref{EEDb})
\begin{align}
E - E_{\infty}=& \iO u\ln\Bigl(\frac{u}{\aq}\Bigr)\,dx +\iO v\ln\Bigl(\frac{v}{\bq}\Bigr)\,dx+\iO w\ln\Bigl(\frac{w}{\cq} \Bigr)\,dx\nonumber\\
&+ q(\ov{u},\ai)+q(\ov{v},\bi))+q(\ov{w},\ci). 
\label{CKoneb}
\end{align}
In order to Taylor-expand the last three terms in \fref{CKoneb}, we observe first that
$$
q'(x_{\infty},x_{\infty}) = 0, \qquad \text{and}\qquad q''(x,x_{\infty}) = 
\frac{x_{\infty}}{x}.
$$
Since the non-negativity of the solutions and the conservation laws imply the natural 
a-prior bounds
$$
0\le\ov{u}, \ai \le \frac{M_1}{\gamma}, \qquad  
0\le\ov{v}, \bi \le \frac{M_2}{\gamma}, \qquad
0\le\ov{w}, \ci \le \min\left\{\frac{M_1}{\alpha}, \frac{M_2}{\beta}\right\},
$$
Taylor expansion yields the following estimate with $\theta \in (\ov{u},\ai)$
\begin{align*}
q(\ov{u},\ai) &= q(\ai,\ai) + q'(\ai,\ai)(\ov{u}-\ai) + q''(\theta,\ai)\frac{(\ov{u}-\ai)^2}{2},\\
&\ge \frac{\ai \gamma}{M_1}\frac{(\ov{u}-\ai)^2}{2}.
\end{align*}
Similar estimates hold for $q(\ov{v},\bi)$ and $q(\ov{w},\ci)$ and we obtain
\begin{multline*}
q(\ov{u},\ai) + q(\ov{v},\bi) + q(\ov{w},\ci)\\
\ge C(\alpha,\beta,\gamma,M_1,M_2) 
\left(|\ov{u}-\ai|^2+ |\ov{v}-\bi|^2 + |\ov{w}-\ci|^2\right)\label{reac}
\end{multline*}
for a constant $C(\alpha,\beta,\gamma,M_1,M_2)>0$ depending only on 
$\alpha$, $\beta$, $\gamma$, $M_1$ and $M_2$. 

Secondly, considering the integral terms on the right-hand side of \fref{CKoneb}, we estimate with
the classical Cziszar-Kullback-Pinsker inequality (Cf. \cite{cziszar}) 
\begin{equation*}
\iO u\ln\Bigl(\frac{u}{\ov{u}} \Bigr)\,dx \ge \frac{1}{2\ov{u}}\|u-\ov{u}\|_{L^1(\Omega)}^2,
\end{equation*}
and analog for $v$ and $w$, 
for which we have again $\ov{u},\ov{v},\ov{w}\le C(\alpha,\beta,\gamma,M_1,M_2)$. 

Altogether, after using Young's inequality to estimate
$\|u-\ai\|_{L^1(\Omega)}^2 \le 2\|u-\ov{u}\|_{L^1(\Omega)}^2+2|\ov{u}-\ai|^2$, we obtain 
\begin{equation*}
E - E_{\infty} \ge C
\left(\|u-\ai\|_{L^1(\Omega)}^2 + \|v-\bi\|_{L^1(\Omega)}^2 + \|w-\ci\|_{L^1(\Omega)}^2\right),
\end{equation*}
for a constant $C=C(\alpha,\beta,\gamma,M_1,M_2)$.
This ends the proof of Proposition \ref{ABCDCK}.
\end{proof}

\bigskip
We now are in a position to state the proof of Theorem \ref{tt2}. 
\begin{proof}[\bf{Proof of Theorem \ref{tt2}}]\hfill\\ 
{
The entropy dissipation law \eqref{entroDiss} yields for the relative entropy with respect to the equilibrium, i.e.  
$\psi(t) := E(u,v,w)(t)-E(\ai,\bi,\ci)$, the following inequality for a.e. $0\le t_0 \le t_1<T$
\begin{equation*}
\psi(t_1) \le \psi(t_0) -\int_{t_0}^{t_1} D(s), \qquad \text{for a.e.}\   0\le t_0 \le t_1<T,
\end{equation*}
which rewrites with Proposition \ref{ABCDDH}, i.e. $D\ge K\psi(s)$ into 
\begin{equation}\label{one}
\psi(t_0) \ge\psi(t_1) + K  \int_{t_0}^{t_1} \psi(s), \qquad \text{for a.e.}\   0\le t_0 \le t_1<T.
\end{equation}
We can now apply a Gronwall argument as stated in \cite{Bee} (see also \cite{Wil} 
for the generalisation where inequality \eqref{one} holds only almost everywhere for an integrable function 
$\psi\in L^1([0,T))$ and obtain exponential convergence of the relative entropy $\psi$, i.e. 
\begin{equation}\label{two}
\psi(t_1) \le\psi(t_0)\, e^{-K(t_1-t_0)}, \qquad \text{for a.e.}\   0\le t_0 \le t_1<T.
\end{equation}
For the convenience of the reader, we shall recall the proof of \eqref{two} in the following: 
First, we perform in \eqref{one} the change of variables $t=-r$ and $\psi(-r)=\tilde{\psi}(r)$ and 
obtain
\begin{equation}\label{oneone}
\tilde\psi(r_0) \ge\psi(t_1) + K  \int_{-t_1}^{r_0} \tilde\psi(s), \qquad \text{for a.e.}\   -T< -t_1 \le r_0\le 0.
\end{equation}
Then, we define $\Psi(r)=\int_{-t_1}^{r} \tilde\psi(r)$ and calculate with
$\dot{\Psi}(r) = \tilde\psi(r) \ge \psi(t_1) + K \Psi(r)$ the well defined derivative
\begin{align*}
\frac{d}{dr}\left(\Psi(r) e^{-K(r+t_1)}\right)& \ge \left(\psi(t_1) + K \Psi(r)\right) e^{-K(r+t_1)} - K \Psi(r) e^{-K(r+t_1)}\\
&\ge \psi(t_1)\, e^{-K(r+t_1)}.
\end{align*}
Then, integration over $[-t_1,r_0]$ and division by  $e^{-K(r_0+t_1)}$ yields 
$$
\Psi(r_0)  \ge \frac{\psi(t_1)}{K}\left(e^{K(r_0+t_1)}- 1\right),
$$
and further with \eqref{oneone} and $\int_{-t_1}^{r_0} \tilde\psi(r)  = \Psi(r_0)$
$$
\tilde\psi(r_0)  \ge \psi(t_1)\, e^{K(r_0+t_1)}.
$$
Then, return to the original variables  $t_0=-r_0$ and $\psi(-r_0)=\tilde{\psi}(r_0)$ yields \eqref{one} and thus by setting $t_0=0$
\begin{equation}
E(u,v,w)-E(\ai,\bi,\ci)\le \left[E(u_0,v_0,w_0)-E(\ai,\bi,\ci)\right]\,e^{-Kt}.
\label{ABCDGron}
\end{equation}
Finally, the estimate \eqref{res2} follows from \eqref{ABCDGron} by applying the  
Cziszar-Kullback type inequality in Proposition \ref{ABCDCK}.
}
\end{proof}

\end{document}